\documentclass[10pt,draft]{article}
\usepackage{amsmath,amscd}
\usepackage{amssymb,latexsym,amsthm}
\usepackage{color}
\usepackage[spanish,english]{babel}
\usepackage{amssymb}
\usepackage{color}
\usepackage{amsmath,amsthm,amscd}
\usepackage[latin1]{inputenc}
\usepackage{lscape}
\usepackage{fancyhdr}
\usepackage{amsfonts}
\usepackage{pb-diagram}

\numberwithin{equation}{section}
\newtheorem{theorem}{Theorem}[section]

\newtheorem{definition}[theorem]{Definition}
\newtheorem{proposition}[theorem]{Proposition}
\newtheorem{lemma}[theorem]{Lemma}

\newtheorem{corollary}[theorem]{Corollary}

\theoremstyle{definition}

\newtheorem{example}[theorem]{Example}

\newtheorem{remark}[theorem]{Remark}

\title{\textbf{Some homological properties\\ of
skew $ PBW $ extensions\\
arising in non-commutative algebraic geometry}}
\author{Oswaldo Lezama\\
\texttt{jolezamas@unal.edu.co}
\\Helbert Venegas
\\ Seminario de Álgebra Constructiva - SAC$^2$\\ Departamento de Matemáticas\\ Universidad Nacional de
Colombia, Sede Bogotá}
\date{}
\begin{document}
\maketitle
\begin{abstract}
\noindent In this short paper we study for the skew $PBW$ (Poincaré-Birkhoff-Witt) extensions some
homological properties arising in non-commutative algebraic geometry, namely, Auslander-Goresntein
regularity, Cohen-Macaulayness and strongly noetherianity. Skew $PBW$ extensions include a
considerable number of non-commutative rings of polynomial type such that classical $PBW$
extensions, quantum polynomial rings, multiplicative analogue of the Weyl algebra, some Sklyanin
algebras, operator algebras, diffusion algebras, quadratic algebras in $3$ variables, among many
others. For some key examples we present the parametrization of its point modules.

\bigskip


\bigskip

\end{abstract}

\section{Introduction}

In the study of non-commutative algebraic geometry many important classes of non-commutative rings
and algebras have been investigated intensively in the last years. Actually, the non-commutative
algebraic geometry consists in generalizing some classical results of commutative algebraic
geometry to some special non-commutative rings and algebras. Probably the most important types of
such rings are the Artin-Schelter regular algebras, Auslander regular algebras, the Auslander
Gorenstein rings, the Cohen-Macaulay rings, and the strongly Noetherian algebras. In
non-commutative algebraic geometry, Artin-Schelter regular algebras are the analog of the
commutative polynomials in commutative algebraic geometry, and in addition, all examples studied of
Artin-Schelter regular algebras are Auslander regular, Auslander-Gorenstein, Cohen-Macaulay, and
strongly Noetherian. In the present paper we are interested in the Auslander-Gorenstein,
Cohen-Macaulay, and strongly Noetherian conditions for skew $PBW$ extensions (the Auslander regular
condition was investigated in \cite{Reyes}). Skew $PBW$ extensions are rings of polynomial type and
generalize the classical $ PBW $ extensions (see \cite{LezamaGallego}, \cite{Oswaldo}). Many
ring-theoretical and homological properties of skew $ PBW $ extensions have been studied in the
last years (Hilbert basis theorem, global dimension, Krull dimension, Goldie dimension, prime
ideals, etc., see \cite{Acosta1}, \cite{Acosta}, \cite{Oswaldo}, \cite{Reyes}); however, the three
homological conditions mentioned before arising in non-commutative algebraic geometry have not been
considered for the skew $ PBW $ extensions.

In the first and second sections we recall the definition of skew $ PBW $ extensions and some
important examples of this type of non-commutative rings. In the third section we discuss the
Auslander-Gorenstein condition. The next two sections are dedicated to investigate the
Cohen-Macaulay and the strongly Noetherian properties. In the last section we present an
application of the strongly noetherianity to parametrize the point modules of some key skew $PBW$
extensions.

\begin{definition}[\cite{LezamaGallego}]\label{gpbwextension}
Let $R$ and $A$ be rings. We say that $A$ is a \textit{skew $PBW$ extension of $R$} $($also called
a $\sigma-PBW$ extension of $R$$)$ if the following conditions hold:
\begin{enumerate}
\item[\rm (i)]$R\subseteq A$.
\item[\rm (ii)]There exist finitely many elements $x_1,\dots ,x_n\in A$ such $A$ is a left $R$-free module with basis
\begin{center}
${\rm Mon}(A):= \{x^{\alpha}=x_1^{\alpha_1}\cdots x_n^{\alpha_n}\mid \alpha=(\alpha_1,\dots
,\alpha_n)\in \mathbb{N}^n\}$, with $\mathbb{N}:=\{0,1,2,\dots\}$.
\end{center}
The set $Mon(A)$ is called the set of standard monomials of $A$.
\item[\rm (iii)]For every $1\leq i\leq n$ and $r\in R-\{0\}$ there exists $c_{i,r}\in R-\{0\}$ such that
\begin{equation}\label{sigmadefinicion1}
x_ir-c_{i,r}x_i\in R.
\end{equation}
\item[\rm (iv)]For every $1\leq i,j\leq n$ there exists $c_{i,j}\in R-\{0\}$ such that
\begin{equation}\label{sigmadefinicion2}
x_jx_i-c_{i,j}x_ix_j\in R+Rx_1+\cdots +Rx_n.
\end{equation}
Under these conditions we will write $A:=\sigma(R)\langle x_1,\dots ,x_n\rangle$.
\end{enumerate}
\end{definition}
Associated to a skew $PBW$ extension $A=\sigma(R)\langle x_1,\dots ,x_n\rangle$, there are $n$
injective endomorphisms $\sigma_1,\dots,\sigma_n$ of $R$ and $\sigma_i$-derivations, as the
following proposition shows.
\begin{proposition}\label{sigmadefinition}
Let $A$ be a skew $PBW$ extension of $R$. Then, for every $1\leq i\leq n$, there exist an injective
ring endomorphism $\sigma_i:R\rightarrow R$ and a $\sigma_i$-derivation $\delta_i:R\rightarrow R$
such that
\begin{center}
$x_ir=\sigma_i(r)x_i+\delta_i(r)$,
\end{center}
for each $r\in R$.
\end{proposition}
\begin{proof}
See \cite{LezamaGallego}, Proposition 3.
\end{proof}

A particular case of skew $PBW$ extension is when all derivations $\delta_i$ are zero. Another
interesting case is when all $\sigma_i$ are bijective and the constants $c_{ij}$ are invertible. We
recall the following definition (cf. \cite{LezamaGallego}).
\begin{definition}\label{sigmapbwderivationtype}
Let $A$ be a skew $PBW$ extension.
\begin{enumerate}
\item[\rm (a)]
$A$ is quasi-commutative if the conditions {\rm(}iii{\rm)} and {\rm(}iv{\rm)} in Definition
\ref{gpbwextension} are replaced by
\begin{enumerate}
\item[\rm (iii')]For every $1\leq i\leq n$ and $r\in R-\{0\}$ there exists $c_{i,r}\in R-\{0\}$ such that
\begin{equation}
x_ir=c_{i,r}x_i.
\end{equation}
\item[\rm (iv')]For every $1\leq i,j\leq n$ there exists $c_{i,j}\in R-\{0\}$ such that
\begin{equation}
x_jx_i=c_{i,j}x_ix_j.
\end{equation}
\end{enumerate}
\item[\rm (b)]$A$ is bijective if $\sigma_i$ is bijective for
every $1\leq i\leq n$ and $c_{i,j}$ is invertible for any $1\leq i<j\leq n$.
\end{enumerate}
\end{definition}

\begin{definition}\label{1.1.6}
Let $A$ be a skew $PBW$ extension of $R$ with endomorphisms
$\sigma_i$, $1\leq i\leq n$, as in Proposition
\ref{sigmadefinition}.
\begin{enumerate}
\item[\rm (i)]For $\alpha=(\alpha_1,\dots,\alpha_n)\in \mathbb{N}^n$,
$\sigma^{\alpha}:=\sigma_1^{\alpha_1}\cdots \sigma_n^{\alpha_n}$,
$|\alpha|:=\alpha_1+\cdots+\alpha_n$. If
$\beta=(\beta_1,\dots,\beta_n)\in \mathbb{N}^n$, then
$\alpha+\beta:=(\alpha_1+\beta_1,\dots,\alpha_n+\beta_n)$.
\item[\rm (ii)]For $X=x^{\alpha}\in Mon(A)$,
$\exp(X):=\alpha$ and $\deg(X):=|\alpha|$.
\item[\rm (iii)]Let $0\neq f\in A$, $t(f)$ is the finite
set of terms that conform $f$, i.e., if $f=c_1X_1+\cdots +c_tX_t$,
with $X_i\in Mon(A)$ and $c_i\in R-\{0\}$, then
$t(f):=\{c_1X_1,\dots,c_tX_t\}$.
\item[\rm (iv)]Let $f$ be as in {\rm(iii)}, then $\deg(f):=\max\{\deg(X_i)\}_{i=1}^t.$
\end{enumerate}
\end{definition}

The next theorems establish some results for skew $ PBW $ extensions that we will use later, for
their proofs see \cite{Oswaldo}.

\begin{theorem}\label{filteredskew}
Let $ A $ be an arbitrary skew $PBW$ extension of the ring $ R $. Then, $ A $ is a filtered ring with filtration given by
$$F_{m}:=
\begin{cases}
R, & \text{ if } m=0\\
\{f \in A| deg(f) \leq m\}, & \text{  if } m \geq 1
\end{cases}$$
and the corresponding graded ring $ Gr(A) $ is a quasi-commutative skew $ PBW $ extension of $ R $.
Moreover, if $ A $ is bijective, then $ Gr(A) $ is quasi-commutative bijective skew $ PBW $
extension of $ R $.
\end{theorem}
\begin{theorem}\label{skewore}
Let $ A $ be a quasi-commutative skew $ PBW $ extension of a ring $ R $. Then,
\begin{enumerate}
\item $ A $ is isomorphic to an iterated skew polynomial ring of endomorphism type, i.e.,
$$ A \cong R[z_{1};\theta_{1}]\dots[z_{n};\theta_{n}].$$
\item If $ A $ is bijective, then each endomorphism $ \theta_{i} $ is bijective, $ 1\leq i\leq n $.
\end{enumerate}
\end{theorem}
\begin{theorem}[Hilbert Basis Theorem] Let $ A $ be a bijective skew $ PBW $ extension of $ R $. If $ R $ is a left (right) Noetherian ring then $ A $ is also a left (right) Noetherian ring.

\end{theorem}

\subsection{Examples}\label{examples}
In order to understand the importance of results of the next three sections, next we recall some
examples of skew $PBW$ extensions; for more details see \cite{Oswaldo} and \cite{Reyes}. $K$
denotes a field.

\begin{example}[\textbf{$ PBW $ extensions}, \cite{Bell}]
Any $ PBW $ extension is a bijective skew $ PBW $ extension since in this case $ \sigma_{i}=i_{R} $ for each $ 1 \leq i \leq n $, and $ c_{i,j}=1 $ for every $ 1 \leq i,j \leq n $. Thus, for $ PBW $ extensions we have $ A=i(R)\langle x_{1}, \dots, x_{n}\rangle $. Some examples of $ PBW $ extensions are the following:
\begin{enumerate}
\item[(a)] The \textit{habitual polynomial ring} $ A=R[t_{1},\dots,t_{n}] $.
\item[(b)] Any \textit{skew polynomial ring of derivation type} $ A=R[x;\sigma,\delta] $, i.e., with $ \sigma=i_{R} $. In general, any \textit{Ore extension of derivation type} $ R[x_{1};\sigma_{1},\delta_{1}]\cdots [x_{n};\sigma_{n},\delta_{n}] $, i.e., such that $ \sigma_{i}=i_{R} $, for any $ 1 \leq i \leq n $.
\item[(c)] The \textit{Weyl algebra} $ A_{n}(K):=K[t_{1},\dots,t_{n}][x_{1};\partial/\partial t_{1}] \cdots [x_{n};\partial / \partial t_{n}] $. The \textit{extended Weyl algebra} $ B_{n}(K):=K(t_{1},\dots,t_{n})[x_{1};\partial/\partial t_{1}] \cdots [x_{n};\partial / \partial t_{n}] $, where $ K(t_{1},\dots,t_{n}) $ is the field of fractions of $ K[t_{1},\dots,t_{n}] $.
\item[(d)] The \textit{universal enveloping algebra} of a finite dimensional Lie algebra $ \mathcal{U}(\mathcal{G}) $. In this case, $ x_{i}r-rx_{i}=0 $ and $ x_{i}x_{j}-x_{j}x_{i}=[x_{i},x_{j}]\in \mathcal{G}=K+Kx_{1}+\dots+Kx_{n} $, for any $ r \in K $ and $ 1 \leq i,j\leq n $.
\item[(e)] The \textit{crossed product} $ R \ast \mathcal{U}(\mathcal{G}) $, in particular, the tensor product $ A:=R \otimes_{K}\mathcal{U}(\mathcal{G}) $ is a $ PBW $ extension of $ R $, where $ R $ is a $ K $-algebra.
\end{enumerate}
\end{example}
\begin{example}[\textbf{Ore extensions of bijective type}]
Any \textit{skew polynomial ring} $ R[x;\sigma,\delta] $ of bijective type, i.e., with $ \sigma $
bijective, is a bijective skew $ PBW $ extension. In this case we have $ R[x;\sigma,\delta]\cong
\sigma(R)\langle x \rangle $. If additionally $ \delta=0 $, then $ R[x;\sigma] $ is
quasi-commutative. More generally, let $
R[x_{1};\sigma_{1},\delta_{1}]\cdots[x_{n};\sigma_{n},\delta_{n}] $ be an \textit{iterated skew
polynomial ring of bijective type}, i.e., the following conditions hold:
\begin{itemize}
\item for $ 1\leq i \leq n $, $ \sigma_{i} $ is bijective;
\item for every $ r \in R $ and $ 1 \leq i \leq n $, $ \sigma_{i}(r) \in R$,  $ \delta_{i}(r) \in R$;
\item for $ i < j $, $ \sigma_{j}(x_{i})=cx_{i}+d $, with $ c,d\in R $ and $ c $ has a left inverse;
\item for $ i<j $, $ \delta_{j}(x_{i})\in R+Rx_{1}+\cdots+Rx_{n} $,
\end{itemize}
then $ R[x_{1};\sigma_{1},\delta_{1}]\cdots[x_{n};\sigma_{n},\delta_{n}]  $ is a bijective skew $ PBW $ extension. Under those conditions we have
$$ R[x_{1};\sigma_{1},\delta_{1}]\cdots[x_{n};\sigma_{n},\delta_{n}] \cong \sigma(R)\langle x_{1},\dots,x_{n} \rangle. $$
Some concrete examples of Ore algebras of bijective type are the following:
\begin{enumerate}
\item[(a)] The \textit{algebra of $ q- $differential operators} $ D_{q,h}[x,y] $.
\item[(b)] The \textit{algebra of shift operators} $ S_{h} $.
\item[(c)] The \textit{mixed algebra} $ D_{h} $.
\item[(d)] The \textit{algebra for multidimensional discrete linear systems}.

\end{enumerate}
\end{example}

\begin{example}[\textbf{Operator algebras}]
Some important and well-known operator algebras like:
\begin{enumerate}
\item[(a)] \textit{Algebra of linear partial differential operators}.
\item[(b)] \textit{Algebra of linear partial shift operators}.
\item[(c)] \textit{Algebra of linear partial difference operators}.
\item[(d)] \textit{Algebra of linear partial $ q- $dilations operators}.
\item[(e)] \textit{Algebra of linear partial $ q- $differential operators}.
\end{enumerate}
\end{example}
\begin{example}[\textbf{Diffusion algebras}] A \textit{diffusion algebra } $ A $ is generated by $ \{D_{i},x_{i}|1 \leq i \leq n \} $ over $ K $ with relations
\begin{equation*}
x_{i}x_{j}=x_{j}x_{i}, \qquad x_{i}D_{j}=D_{j}x_{i}, \qquad 1 \leq i,j \leq n.
\end{equation*}
\begin{equation*}
c_{ij}D_{i}D_{j}-c_{ji}D_{j}D_{i}=x_{j}D_{i}-x_{i}D_{j}, \qquad i<j, c_{ij},c_{ji} \in K^{*}
\end{equation*}
Thus, $ A \cong \sigma(K[x_{1},\dots,x_{n}])\langle D_{1},\dots, D_{n}\rangle $.
\end{example}
\begin{example}[\textbf{Quantum algebras}] Some important examples of Quantum algebras are the following:
\begin{enumerate}
\item [(a)] \textit{Additive analogue of the Weyl algebra}.
\item[(b)] \textit{Multiplicative analogue of the Weyl algebra}.
\item[(c)] \textit{Quantum algebra} $ \mathcal{U}^{'}(so(3,K)) $.
\item[(d)] \textit{$ 3- $dimensional skew polynomial algebra $ A $}.
\item[(e)] \textit{Dispin algebra} $ \mathcal{U}(osp(1,2)) $.
\item[(f)] \textit{Woronowicz algebra} $ \mathcal{W}_{v}(sl(2,K)) $.
\item[(g)] The \textit{complex algebra} $ V_{q}(sl_{3}(\mathbb{C})) $.
\item[(h)] The \textit{algebra $ U $}.
\item[(i)] The \textit{coordinate algebra of the quantum matrix space } $ M_{q}(2) $.
\item[(j)] \textit{$ q- $Heisenberg algebra}.
\item[(k)] \textit{Quantum enveloping algebra of } $ sl(2,K) $, $ \mathcal{U}_{q}(sl(2,K)) $.
\item[(l)] \textit{Hayashi algebra }$ W_{q}(J) $.
\item[(m)] \textit{Quantum Weyl algebra of Maltsiniotis} $ A_{n}^{q,\lambda} $.
\item[(n)] The \textit{algebra of differential operators} $ D_{q}(S_{q}) $ \textit{on a quantum space} $ S_{q} $.
\item[(0)] \textit{Witten's deformation of } $ \mathcal{U}(sl(2,K)) $.
\item[(p)] \textit{Quantum Weyl algebra $ A_{n}(q,p_{i,j}) $}.
\item[(q)] \textit{Multiparameter quantized Weyl algebra} $ A_{n}^{Q,\Gamma}(K) $.
\item[(r)] \textit{Quantum symplectic space} $ \mathcal{O}_{q}(sp(K^{2n})) $.
\end{enumerate}

\end{example}
\begin{example}[\textbf{Quadratic algebras in 3 variables}] A quadratic algebra in $ 3 $ variables $ A $ is a $ K- $algebra generated by $ x,y,z $ subject to the relations
\begin{align*}
yx&= xy+a_{1}z+a_{2}y^{2}+a_{3}yz+\xi_{1}z^{2},\\
zx&= xz+\xi_{2}y^{2}+a_{5}yz+a_{6}z^{2},\\
zy&= yz+a_{4}z^{2}.
\end{align*}
These algebras are examples of bijective skew $ PBW $ extensions.

\end{example}
\begin{example}[$ n- $\textit{multiparametric skew quantum space} over $ R $]
Let $ R $ be a ring with a fixed matrix of parameters $  \textbf{q} := [q_{ij}]\in M_{n}(R)$, $ n\geq 2 $, such that $ q_{ii}=1=q_{ij}q_{ji} $ for every $ 1\leq i, j\leq n $, and suppose also that it is given a system $ \sigma_{1}, \dots,\sigma_{n} $ of automorphisms of $ R $. The quasi-commutative bijective skew $ PBW $ extension $ R_{q,\sigma}[x_{1},\dots,s_{n}] $ defined by
\[x_{j}x_{i}=q_{ij}x_{i}x_{j}, \quad x_{i}r=\sigma_{i}(r)x_{i}, \quad r \in R, \quad 1 \leq i,j \leq n,  \]
is called the $ n- $\textit{multiparametric skew quantum space} over $ R $. When all automorphisms of the extension are trivial, i.e., $ \sigma_{i}=i_{R} $, $ 1\leq i\leq n $ it is called $ n- $\textit{multiparametric quantum space} over $ R $. If $ R=K $ is a field, then $ K_{q,\sigma}[x_{1},\dots,x_{n}] $ is  called $ n- $\textit{multiparametric skew quantum space}, and the case particular case $ n=2 $ is called \textit{skew quantum plane}, for trivial automorphisms we have the $ n- $\textit{multiparametric quantum space} and the \textit{quantum plane}.\label{key}
\end{example}

\begin{example}[Skew quantum polynomials]\label{skewquantum}
Let $ R $ be a ring with a fixed matrix of parameters $  \textbf{q} := [q_{ij}]\in M_{n}(R)$, $ n\geq 2 $, such that $ q_{ii}=1=q_{ij}q_{ji} $ for every $ 1\leq i, j\leq n $, and suppose also that it is given a system $ \sigma_{1}, \dots,\sigma_{n} $ of automorphisms of $ R $. The ring of \textbf{skew quantum polynomials} over $ R $, denoted by $ Q_{q,\sigma}^{r,n}:=R_{q,\sigma}[x_{1}^{\pm 1},\dots,x_{r}^{\pm1},x_{r+1},\dots,x_{n}] $, is defined as follows:
\begin{enumerate}
\item[\rm (i)] $ R \subseteq R_{q,\sigma}[x_{1}^{\pm 1},\dots,x_{r}^{\pm1},x_{r+1},\dots,x_{n}] $.
\item[\rm (ii)] $ R_{q,\sigma}[x_{1}^{\pm 1},\dots,x_{r}^{\pm1},x_{r+1},\dots,x_{n}] $ is a free left $ R- $module with basis
\[ \{x_{1}^{\alpha_{1}}\cdots x_{n}^{\alpha_{n}}|\alpha_{i}\in \mathbb{Z} \text{ for }1 \leq i \leq r \text{ and } \alpha_{i}\in \mathbb{N} \text{ for } r+1\leq i \leq n\}.\]
\item[\rm (iii)] The elements $ x_{1},\dots,x_{n} $ satisfy the defining relations
\[x_{i}x_{i}^{-1}=1=x_{i}^{-1}x_{i}, \quad 1 \leq i\leq r,\]
\[x_{j}x_{i}=q_{ij}x_{i}x_{j}, \quad x_{i}r=\sigma_{i}(r)x_{i}, \quad r \in R, \quad 1 \leq i, j \leq n.\]
\end{enumerate}
$ Q_{q,\sigma}^{r,n}$  can be viewed as a localization of a skew $ PBW $ extension or can be defined as a quasi-commutative bijective skew $ PBW $ extension of the $ r- $multiparameter quantum torus (see \cite{Oswaldo}).
\end{example}
\begin{example}[Sklyanin algebra]\label{Sklyanin}
Let $ a,b,c \in K $, then
\[S=K\langle x,y,z\rangle/\langle ayx+bxy+cz^{2},axz+bzx+cy^{2},azy+byz+cx^{2}\rangle\]
is skew $ PBW $ extension if $ c=0 $ and $ a,b\neq 0 $, thus $ S=\sigma(K)\langle x,y,z\rangle $:
\[yx=-\frac{b}{a}xy, \quad zx=-\frac{a}{b}xz, \quad zy=-\frac{b}{a}yz.\]
\end{example}

\begin{example}[3-dimensional skew polynomial algebra]
$\mathcal{A}$ is the $K$-algebra generated by the variables $x,y,z$ restricted to relations
\begin{center}
$yz-\alpha zy=\lambda; zx-\beta xz=\mu; xy-\gamma yx=\nu$;
\end{center}
with $\lambda,\mu,\nu\in K+Kx+Ky+Kz$ and $\alpha,\beta,\gamma\in K-\{0\}$.
\end{example}

\section{Auslander regularity conditions}

\noindent In this section we will study the Auslander regularity and the Auslander-Gorenstein
conditions for skew $ PBW $ extensions. The first condition was studied by Björk in \cite{Bjorkj},
\cite{Bjork} and by Ekström in \cite{Ekstrom} for filtered Zariski rings and Ore extensions;  in
\cite{Reyes} was studied the Auslander condition for skew $PBW$ extensions. We will consider the
second condition, but for completeness, we will integrate the first one in the statements of the
results below.
\begin{definition}
Let $ A $ be a ring.
\begin{enumerate}
\item[\rm (i)] The grade $j(M)$ of a left (or right) $A-$module $M$ is defined by
$$j(M) := min\{i|Ext^{i}_{A}
(M, A) \neq  0\}$$
or $ \infty $ if no such $ i $ exists.

\item[\rm (ii)]
 $A$ satisfies the \textit{Auslander} condition if for every Noetherian left (or right) $A-$module $M$ and for all $i \geq 0$
$$j(N)  \geq i \text{ for all submodules } N \subseteq Ext^{i}_{A}(M, A).$$
\item [\rm (iii)] $A$ is \textit{Auslander-Gorenstein} (AG) if $A$ is Noetherian (i.e., two-sided Noetherian), which satisfies the
Auslander condition, $id(_{A}A)< \infty$, and $id(A_{A})<\infty$.
\item[\rm (iv)] $A$ is \textit{Auslander regular} (AR) if it is AG and $gld(A) < \infty$.
\end{enumerate}
\end{definition}
\begin{remark}
\begin{enumerate}
\item[\rm (i)] We recall that if a ring $ A $ is Noetherian, then $ gld(A):=rgld(A)=lgld(A) $.
\item[\rm (ii)] If $ A $ is Noetherian and if $ id(_{A}A)<\infty $ and $ id(A_{A})<\infty $, then $ id(_{A}A)=id(A_{A}) $ (see \cite{Zaks}).
\end{enumerate}
\end{remark}
\begin{definition}
Let $ A $ be a filtered ring with filtration $ \{F_{n}(A)\}_{n \in \mathbb{Z}} $.
\begin{enumerate}
\item The Rees ring associated to $ A $ is a graded ring defined by
$$\widetilde{A}:=\bigoplus_{n \in \mathbb{Z}}F_{n}(A).$$
\item The filtration $ \{F_{n}(A)\}_{n \in \mathbb{Z}} $ is left (right) Zariskian, and $ A $ is called a left
(right) Zariski ring, if $ F_{-1}(A)\subseteq Rad(F_{0}(A)) $ and the associated Rees ring $ \widetilde{A}  $ is left (right) Noetherian.
\end{enumerate}
\end{definition}
The following result is a characterizations of the Zariski property (see \cite{Li2}).
\begin{proposition}\label{Caracterizacionzariski}
Let $ A $ be a $ \mathbb{N} $-filtered ring such that $ Gr(A) $ is left (right)Noetherian. Then, $ A $ is left (right) Zariskian.
\end{proposition}
\begin{proposition}\label{GradedZariski}
Let $ A $ be a left and right Zariski ring. If its associated graded ring $ Gr(A) $ is $ AG $, respectively $ AR $, then so too is $ A $.
\end{proposition}
\begin{proof}
See \cite{Bjork}, Theorem 3.9.
\end{proof}
\begin{proposition}\label{AGore}
If $ A $ is $ AG $, respectively $ AR $, then the skew polynomial ring $ A[x;\sigma, \delta] $ with $ \sigma $ bijective is also $ AG $, respectively $ AR $.
\end{proposition}
\begin{proof}
\cite{Ekstrom}, Theorem 4.2.
\end{proof}
The following proposition states that the $ AG $ ($ AR $) conditions are preserved under arbitrary localizations.
\begin{proposition}\label{OreAG}
Let $ A $ be an $ AG $ ring, respectively AR, and $ S $ a multiplicative Ore set of regular elements of $ A $. Then so too is  $ S^{-1}A $ (and also $ AS^{-1} $).
\end{proposition}
\begin{proof}
 See \cite{Ajitabh}, Proposition 2.1.
\end{proof}

\begin{lemma}\label{skewZariski}
If $ A $ is a bijective skew $PBW$ extension of a Noetherian ring $ R $, then $ A $ is a left and right Zariski ring.
\end{lemma}
\begin{proof}
Since $ A $ is $ \mathbb{N}- $filtered, $ 0=F_{-1}(A)\subseteq Rad(F_{0}(A))=Rad(R) $. By Theorem \ref{skewore}, $ Gr(A) $ is isomorphic to an iterated skew polynomial ring $ R[z_{1};\theta_{1}]\cdots [z_{n};\theta_{n}] $, with $ \theta_{i} $ is bijective, $ 1 \leq i \leq n $. Whence $ Gr(A) $ is Noetherian. Proposition \ref{Caracterizacionzariski} says that $ A $ is a left and right Zariski ring.
\end{proof}
\begin{theorem}\label{skewAG}
Let $ A $ be a bijective skew $PBW$ extension of a ring $R$ such that $R$ is $AG$, respectively $AR$, then so too is $A$.
\end{theorem}
\begin{proof}
Acoording to Theorem \ref{filteredskew}, $ Gr(A) $ is a quasi-commutative skew $ PBW $ extension, and by the hypothesis, $ Gr(A) $ is also bijective. By Theorem \ref{skewore}, $ Gr(A) $ is isomorphic to an iterated skew polynomial ring $ R[z_{1};\theta_{1}]\dots [z_{n};\theta_{n}] $ such that each $ \theta_{i} $ is bijective, $ 1 \leqslant i \leqslant n $. Proposition \ref{AGore} says that $ Gr(A) $ is $ AG $, respectively $ AR $. From Lemma \ref{skewZariski}, $ A $ is a left and right Zariski ring, so by Proposition \ref{GradedZariski}, $ A $ is $ AG $, respectively $ AR $.
\end{proof}

\begin{corollary}\label{AGquantum}
If $ R $ is $AG$, respectively $AR$, then the ring of skew quantum polynomials
$$Q_{q,\sigma}^{r,n}:=R_{q,\sigma}[x_{1}^{\pm 1},\dots,x_{r}^{\pm1},x_{r+1},\dots,x_{n}]$$
is $AG$, respectively $AR$.

\end{corollary}
\begin{proof}
Let $ R $ be $ AG $; according to Example \ref{skewquantum}, $ Q_{q,\sigma}^{r,n}(R) $ is a
localization of a bijective skew $ PBW $ extension $ A $ of the ring $ R $ by a multiplicative Ore
set of regular elements of $ A $. From Theorem \ref{skewAG} and Proposition \ref{OreAG} we get that
$ Q_{q,\sigma}^{r,n}(R) $ is $ AG $. If $ R $ is $ AR $, then $ R $ is $ AG $ and $ gld(R) < \infty
$, then $ gld(Q_{q,\sigma}^{r,n}(R))<\infty $ and $  Q_{q,\sigma}^{r,n}(R)  $ is $ AG $, so $
Q_{q,\sigma}^{r,n}(R) $ is $ AR $.
 \end{proof}

\begin{example}\label{AGexample}
All examples in the section \ref{examples} are $ AG $ rings (respectively $ AR $)  assuming that the ring of coefficients ($ R $ or $ K $) is $ AG $ ($ AR $). If the ring the coefficients is a field, then of course, it is an $ AG $ ring (respectively $ AR $).
\end{example}

\section{Cohen-Macaulayness}

In this section we study the Cohen- Macaulay property for skew $PBW$ extensions.
\begin{definition}
Let $ A $ be an algebra over a field $ K $, we say that $ A $ is \textit{Cohen-Macaulay} (CM) with respect to  the classical Gelfand-Kirillov dimension  if:
\[GKdim(A)=j_{A}(M)+GKdim(M)\]
for every non-zero Noetherian $ A- $module $ M $.
\end{definition}
Recall that if $A$ is a $K$-algebra, then the \textit{classical Gelfand-Kirillov dimension} of $A$
is defined by
\begin{equation}\label{equ17.2.2}
{\rm GKdim}(A):=\sup_{V}\overline{\lim_{n\to \infty}}\log_n\dim_K V^n,
\end{equation}
where $V$ vanishes over all frames of $A$ and
$V^n:=\, _K\langle v_1\cdots v_n| v_i\in V\rangle$ (a \textit{frame} of $A$ is a finite dimensional $K$-subspace of $A$ such that
$1\in V$; since $A$ is a $K$-algebra, then $K\hookrightarrow A$, and hence, $K$ is a frame of $A$ of dimension $1$). For a $ K- $algebra $ A $,
an automorphism $ \sigma $ of $ A $ is said to be \textit{locally algebraic} if for any $ a\in A $ the set $ \{\sigma^{m}(a)|m\in \mathbb{N}\} $ is cointained in a finite dimensional subspace of $ A $.\\

The classical Gelfand-Kirillov dimension for skew $PBW$ extensions was studied in \cite{milton} (in
\cite{LezamaLatorre} has been studied the Gelfand-Kirillov dimension of skew $PBW$ extensions
assuming that the ring $R$ of coefficients is a left Noetherian domain).
\begin{proposition}
Let $R$ be a $K-$algebra with a finite dimensional generating subspace $V$ and let
$A=\sigma(R)\langle x_{1},\dots,x_{n}\rangle$ be a bijective skew $PBW$ extension of $R$. If
$\sigma_{n}(V)\subseteq V$ or $\sigma_{n}$ is locally algebraic, then
\[GKdim(A)=GKdim(R)+n.\]
\end{proposition}
The following proposition in the classical case is also known.
\begin{proposition}
Let $ A $ be a $ K- $algebra with a finite filtration $ \{A_{i}\}_{i \in \mathbb{Z}} $ such that $ Gr(A) $ is finitely generated. Then $ GKdim(_{Gr(A)}Gr(A)) =GKdim(_{A}A)$.
\end{proposition}
\begin{proof}
See \cite{Krause}, Proposition 6.6.
\end{proof}

\begin{proposition}
Suppose that $A$ is a left and right Zariskian ring, and $Gr(A)$ is AG. Then
$j_{A}(M)=j_{Gr(A)}(Gr(M))$ for every non-zero Noetherian $A-$module $M$.
\end{proposition}
\begin{proof}
If $M$ is a finitely generated $A-$module and if $\{F_{n}(M)\}$ is a good filtration on $M$ then in
general $j_{A}(M)\leq j_{Gr(A)}(Gr(M))$, but  if $A$ is left and right Zariski ring and $Gr(A)$ is
$AG$ then $j_{A}(M)\geq j_{Gr(A)}(Gr(M))$, so $j_{A}(M)= j_{Gr(A)}(Gr(M))$ (see \cite{Bjork}, Proof
of Theorem 3.9).
\end{proof}
Now, it is  possible to prove the following proposition.
\begin{proposition}\label{CMgraded}
Let $ A $ be a left and right Zariski ring with finite filtration and such that $Gr(A)$ is AG. If $
Gr(A) $ is CM, then so too is $ A $.
\end{proposition}
\begin{proof}
Let $ M $ be a noetherian $ A- $module, then
\begin{align*}
GKdim(A)&=GKdim(Gr(A))\\
&=GKdim(_{Gr(A)}Gr(M))+j_{Gr(A)}(M)\\
&=GKdim(_{A}M)+j_{A}(M).
\end{align*}
Therefore $ A $ is $CM$.
\end{proof}
\begin{proposition}\label{CMore}
Suppose that $R$ is  $AR$ $(AG)$ and $CM$ ring. Let $R[x;\sigma,\delta]$ be an Ore extension with
$\sigma$ bijective. If $R=\oplus_{i\geq 0}R_{i}$ is a connected graded $K-$algebra (i.e.,
$R_{0}=K$) such that $\sigma(R_{i})\subseteq R_{i}$ for each $i\geq 0$. Then $R[x;\sigma,\delta]$
is $CM$.
\end{proposition}
\begin{proof}
See \cite{Levasseur2}, Lemma, Part \rm (ii).
\end{proof}

\begin{definition}
Let $A$ be a $K$-algebra,  it is said that $ x\in A $ is a local normal element if for every frame
$ V\subset A $, there is a frame $ V'\supset V $ such that $ xV'=V'x $.
\end{definition}
It is clear that every central element is local. The next proposition says that $CM$ property is
preserved under certain localizations.
\begin{proposition}\label{CMlocalization}
Let $ A $ be an $ AG $ ring, and $ S $ a multiplicatively closed set of local normal elements in $
A $. If $ A $ is $CM$, so is $S^{-1} A$.
\end{proposition}
\begin{proof}
cf. \cite{Ajitabh}, Theorem 2.4.
\end{proof}
\begin{theorem}\label{CMskew}
Let $A$ be a bijective skew $PBW$ extension of a ring $R$ such that $R$ is $AG$, $CM$, and
$R=\sum_{i\geq 0}\oplus R_{i}$ is a connected graded $K-$algebra such that
$\sigma_j(R_{i})\subseteq R_{i}$ for each $i\geq 0$ and $1\leq j\leq n$, then $A$ is $CM$.
\end{theorem}
\begin{proof}
From Theorem \ref{filteredskew} it is clear that $A$ is a $K-$algebra with a finite filtration and
$Gr(A)$ is a quasi-commutative skew $PBW $ extensions, and by the hypothesis, $Gr(A)$ is also
bijective. By Theorem \ref{skewore}, $ Gr(A) $ is isomorphic to an iterated skew polynomial ring $
R[z_{1};\theta_{1}]\dots [z_{n};\theta_{n}] $ such that each $ \theta_{i} $ is bijective, $ 1
\leqslant i \leqslant n $ (according to the proof of Theorem \ref{skewore} in \cite{Oswaldo},
$\theta_j(r)=\sigma_j(r)$ for every $r\in R$). Proposition \ref{AGore} says that $ Gr(A) $ is $ AG
$. From Lemma \ref{skewZariski}, $ A $ is a left and right Zariski ring, and by Proposition
\ref{CMore} $Gr(A)$ is $CM$, so by Proposition \ref{CMgraded} $A$ is $CM$.
\end{proof}

\begin{corollary}\label{CMquantum}
Let $R=\sum_{i\geq 0}\oplus R_{i}$ be a connected graded $K-$algebra such that
$\sigma_j(R_{i})\subseteq R_{i}$ for every $i\geq 0$ and all $\sigma_{j}$ in definition of skew
quantum polynomials $Q_{q,\sigma}^{r,n}$. If $R$ is  $AG$ and $CM$, then $Q_{q,\sigma}^{r,n}$ is
$CM$.
\end{corollary}
\begin{proof}
By Example \ref{skewquantum}, $Q_{q,\sigma}^{r,n}$ is a localization of a bijective skew $PBW$
extension $A$
of $R$ by a multiplicative Ore set of regular elements of $A$, and the multiplicative set generated by $x_{1},\dots,x_{r}$ consists of monomials, which are local normal elements. 
 From Theorem \ref{CMskew}  and
Proposition \ref{CMlocalization} we get that $Q_{q,\sigma}^{r,n}$ is $CM$.
\end{proof}

\section{Strongly noetherian algebras}

Now we will consider the strongly noetherian property for skew $PBW$ extensions, this condition was
studied by Artin, Small and Zhang in \cite{Artin2}, and appears naturally in the study of point
modules in non-commutative algebraic geometry (see \cite{Artin1}, \cite{Artin3} and
\cite{Rogalski}).
\begin{definition} Let $K$ be a commutative ring and let $A$ be a left Noetherian $K$-algebra. We
say that $A$ is left strongly Noetherian if for any commutative Noetherian $K$-algebra $C$,
$C\otimes_{K}A$
 is left Noetherian.
\end{definition}
Some examples of strongly Noetherian algebras include Weyl algebras, Sklyanin algebras over a field
$ K $ and universal enveloping algebras of finite dimensional Lie algebras (see \cite{Artin2},
Corollaries 4.11 and 4.12). Moreover, all known examples of Artin-Shelter regular algebras are
strongly Noetherian, is an open question if every Artin-Shelter regular algebra is strongly
Noetherian.

\begin{proposition}\label{SN}
Let $ K $ be a commutative ring and let $ A $ be a $ K- $algebra.
\begin{enumerate}
\item[\rm (i)] If $ A $ is left strongly Noetherian, then $A[x;\sigma,\delta]$ is left strongly Noetherian when $ \sigma $ is bijective.
\item[\rm (ii)] If $ A $ is $ \mathbb{N}- $filtered and $Gr(A)$ is left strongly Noetherian, then $ A $ is left strongly Noetherian.
\item[\rm (iii)] If $ S $ is a multiplicative Ore set of regular elements of $ A $, then $ S^{-1}A $ is left strongly Noetherian.
\end{enumerate}
\end{proposition}
\begin{proof}
See \cite{Artin2}, Proposition 4.1. and Proposition 4.10.
\end{proof}
With the previous result we get the main result of the present section.
\begin{theorem}\label{skewstrongly}
Let $K$ be a commutative ring and let $A = \sigma(R)\langle x_{1},\dots,x_{n}\rangle$ be a bijective skew $PBW$ extension of a left strongly Noetherian $K$-algebra $R$. Then $A$ is left strongly Noetherian.
\end{theorem}
\begin{proof}
According to Theorem \ref{filteredskew}, $ Gr(A) $ is a quasi-commutative skew $ PBW $ extension, and by the hypothesis, $ Gr(A) $ is also bijective. By Theorem \ref{skewore}, $ Gr(A) $ is isomorphic to an iterated skew polynomial ring $ R[z_{1},\theta_{1}]\dots [z_{n},\theta_{n}] $ such that each $ \theta_{i} $ is bijective, $ 1 \leq i \leq n $. Proposition \ref{SN} part \rm (i) says that $ Gr(A) $ is left strongly Noetherian, and by part \rm (ii)  $ A $ is left strongly Noetherian.
\end{proof}

\begin{corollary}\label{SNquantum}
Let $ R $ be a left strongly Noetherian $ K $-algebra, then the ring of skew quantum polynomials
$$Q_{q,\sigma}^{r,n}:=R_{q,\sigma}[x_{1}^{\pm 1},\dots,x_{r}^{\pm1},x_{r+1},\dots,x_{n}]$$
is left strongly Noetherian.
\end{corollary}
\begin{proof}
According to Example \ref{skewquantum}, $ Q_{q,\sigma}^{r,n}(R) $ is a localization of a bijective
skew $ PBW $ extension $ A $ of the ring $ R $ by a multiplicative Ore set of regular elements of $
A $. From Theorem \ref{skewstrongly} and Proposition \ref{SN} part \rm (iii) we get that $
Q_{q,\sigma}^{r,n}(R) $ is left strongly Noetherian.
\end{proof}
\begin{example}\label{SNexample}
All examples in  Section \ref{examples} are left strongly Noetherian assuming that the ring of coefficients is left strongly Noetherian $ K- $algebra.
\end{example}

\section{Point modules of some skew $PBW$ extensions}

As application of strongly Noetherian algebras, we present next some examples of finitely graded
algebras that are bijective skew $PBW$ extensions and such that they have nice spaces parametrizing
its point modules. It is important to remark that Artin, Tate and Van den Bergh studied point
modules in order to complete the classification of Artin-Shelter regular algebras of dimension $ 3
$ (see \cite{Artin1}).

Recall that a $K$-algebra $ A $ is \textit{finitely graded} if it is connected $ \mathbb{N}-
$graded and finitely generated as a $ K- $algebra.

\begin{definition}
Let $ A $ be a finitely graded $ K- $algebra that is generated in degree $ 1 $. A point module for $ A $ is a graded left  module $ M $ such that $ M $ is cyclic, generated in degree $ 0 $, and  $ dim_{K}(M_{n})=1 $ for all $ n \geq 0 $.
\end{definition}

 If $ A $ is commutative, then its point modules naturally correspond to the (closed) points of the scheme $ proj(A) $. Similarly,
  many non-commutative graded rings also have nice parameter spaces of point modules. For example, the point modules of the quantum
  plane or Jordan plane are parametrized by $ \mathbb{P}^{1} $, i.e., there exists a bijective correspondence between the projective space
  $ \mathbb{P}^{1} $ and the collection of isomorphism classes of point modules of the quantum plane or Jordan plane (see \cite{Rogalski}, Example 3.2. ).
 The following proposition states that for finitely graded strongly Noetherian algebras is possible to find a
 projective scheme that parametrizes the set of its point modules.

\begin{proposition}[\cite{Artin3}, Corollary [4.12]\label{E4.12}
Let $ A $ be a finitely graded $ K- $algebra which is generated in degree $ 1 $. If $ A $ is
strongly Noetherian, then the point modules of $ A $ are naturally parametrized by a commutative
projective scheme over $ K $.
\end{proposition}

The following bijective skew $PBW$ extensions are finitely graded $K-$algebras generated in degree
1:
\begin{enumerate}
\item[\rm (i)] Polynomial ring, $K[x_{1},\dots,x_{n}]$.
\item[\rm (ii)] Skew polynomial ring $K[x;\sigma]$ with $\sigma$ bijective. In general, $K[x_{1};\sigma_{1}]\cdots[x_{n};\sigma_{n}]$ with $\sigma_{i}$ bijective for $1\leq i \leq n$.
\item[\rm (iii)] Quantum polynomial ring.
\item[\rm (iv)] Multiplicative analogue of the Weyl algebra.
\item[\rm (v)] The Sklyanin algebra.
\end{enumerate}
According to Theorem \ref{skewstrongly} they are left strongly noetherian and by Proposition
\ref{E4.12} there exists  a commutative projective scheme that parametrizes the set of point
modules of $ A $ .
 \begin{corollary}\label{pointskew}
 Let $A$ be a bijective skew $PBW$ extension  as above, then there exists a commutative projective scheme over $ K $ that parametrizes
 the set of point modules of $ A $.
 \end{corollary}


\end{document}